\documentclass[twocolumn]{aastex631}

\usepackage{graphicx}
\graphicspath{{./}{figs/}}
\usepackage{amsmath}
\usepackage{amssymb}
\usepackage{amsthm}
\usepackage{hyperref}
\usepackage{xspace}
\usepackage{dcolumn}
\usepackage{bbold}
\usepackage{enumitem}
\usepackage{etoolbox}

\newcommand{\ols}{$\omega$-limit set\xspace}
\newcommand{\real}{\mathbb{R}\xspace}
\newtheorem{theorem}{Theorem}
\newtheorem{corollary}{Corollary}[theorem]
\newtheorem{example}{Example}
\newtheorem{remark}{Remark}

\usepackage{amsopn}

\begin{document}

\title[A new description of the Lorenz attractor]{The $\omega$-limit set in a positively invariant compact region\\ and a new description of the Lorenz attractor}

\author{Khalil Zare}
 \affiliation{Retired aerospace engineer and mathematician, Austin, TX, USA \texttt{kalzare@yahoo.com}}

\author[0000-0003-3240-6497]{Steven R. Chesley}
 \affiliation{Jet Propulsion Laboratory, California Institute of Technology, \\
 4800 Oak Grove Dr., Pasadena, CA 91109, USA}


\date{\today}

\begin{abstract}
The $\omega$-limit set in a compact positively invariant region $R \subset \mathbb{R}^n$ has been identified for $n=1$, 2, and 3, with examples in each case. It has been shown that the $\omega$-limit set becomes more complex as $n$ increases from 1 to 3, and we expect this to also be true for $n>3$. Our example for $n=3$ is the {\em Lorenz equations}, for which we have shown that its $\omega$-limit set is a {\em twisted torus}.
\end{abstract}

\section{Introduction}

A solution of differential equations may be considered as a {\em flow} in $\real^n$. If all the flows in or entering a region $R$ remain in $R$ indefinitely as $t\to\infty$ then $R$ is called a {\em positively invariant region}.

The {\em \ols} introduced by  \citet{birkhoff1927} consists of those flows in $R$ that attract the other flows in $R$ as $t\to\infty$. Examples are asymptotically stable equilibria and limit cycles.

In addition, if $R$ is {\em compact} (i.e., closed and bounded) then the \ols of a flow in $R$ is {\em non-empty}, {\em compact}, {\em invariant}, and {\em connected} (see Theorem~\ref{thm:first} in Sec.~\ref{sec:ols_compact}). 

In this paper, we show that in a compact, positively-invariant region $R\subset\real^n$, the \ols can contain 
\begin{enumerate}
  \item [(i)] only {\em equilibria} if $n=1$,\label{list:n1}
  \item [(ii)]\label{list:n2} {\em equilibria}, {\em limit cycles}, and {\em heteroclinic} ({\em homoclinic}) orbits if $n=2$
  \item [(iii)] {\em equilibria}, {\em limit cycles},  {\em heteroclinic} ({\em homoclinic}) orbits, and a surface topologically equivalent to a {\em single or multiple torus} if $n=3$.\label{list:n3}
\end{enumerate}
Notice that the \ols becomes more complex as $n$ increases from 1 to 3. We expect this to be also true for $n>3$, which is not treated in this paper, but left for future studies.

It follows from (ii) that if the region $R$ does not contain any equilibrium then there is a limit cycle in $R$. This is an alternative proof of the {\em Poincar\'e-Bendixson Theorem}.

It follows from (iii) that a lack of an equilibrium in $R$ does not necessarily imply the existence of a limit cycle since the solutions may approach the surface of a torus. This explains why the Poincar\'e-Bendixson Theorem cannot be extended to three dimensions.

To give a three dimensional example, we chose the Lorenz equations concerning the stability of fluid flows in the atmosphere. Using numerical integration, we have shown that the \ols in this case is a twisted torus.

\section{The \texorpdfstring{$\omega$}{omega}-limit set in a compact positively invariant region}\label{sec:ols_compact}

Let $R$ be an open subset of $\real^n$ and 
$$
\vec{f}:R \to \real^n
$$
be a function for which 
\begin{equation}
\frac{d\vec{y}}{dt} = \vec{f}(\vec{y}),\; \vec{y}(0) = \vec{y}_0 \in R   
\end{equation}
has a unique solution $\vec{\Phi}(t, \vec{y}_0)$ that is defined for all $t\in\real$.

Under these assumptions, it follows that 
\begin{equation}
    \vec\Phi(t+\tau, \vec y_0) = \vec\Phi(t, \vec\Phi(\tau, \vec y_0)).
\end{equation}

\noindent{\bf Definition}: The \ols of $\vec y_0 \in R$ is the set 
\begin{multline*}
L_\omega(\vec y_0) = \{ \vec y \in R: \text{there exists an unbounded} \\\text{sequence}\ \{t_k\}\ \text{such that}\lim\ \vec\Phi(t_k, \vec y_o) = \vec y\}. 
\end{multline*}

To discuss the \ols, we use the following theorem from analysis:
\begin{theorem}\label{thm:first}
If the region $R$ is compact (i.e., closed and bounded) and {\em positively-invariant}, then the \ols of a trajectory in $R$ is non-empty, compact, invariant, and connected.
\end{theorem}
\begin{proof}
    For a point $\vec y \in R$ the sequence $\vec\Phi(n,\vec y)$ is contained in the compact set R. By the Bolzano-Weierstrass theorem, this sequence contains a convergent sub-sequence whose limit is, by definition, an element of $L_\omega(\vec y)$. The compactness of $L_\omega(\vec y)$ follows from the fact that all closed subsets of a compact set are also compact.

    To verify that $L_\omega(\vec y)$ is connected, we assume, by contradiction, the existence of two open sets $U$ and $V$ such that 
\begin{gather*}    
L_\omega(\vec y) \cap U \ne \varnothing \\
L_\omega(\vec y) \cap V \ne \varnothing \\
L_\omega(\vec y) \subset U \cup V.
\end{gather*}

Let $\{t_k\} \subset [0,\infty)$ such that $\lim_{k \to \infty}t_k = \infty$, with $\Phi(t_{2k}, y) \in U$ and $\Phi(t_{2k+1}, y) \in V$, for each $k$.

Now the continuous curve 
\begin{equation*}
    \{ \Phi(t,y): t_{2k} < t < t_{2k+1} \}
\end{equation*}
connects a point in $U$ with a point in $V$ and there must be a time $\tau \in (t_{2k}, t_{2k+1})$ such that 
\begin{equation*}
    \Phi(\tau,y) \in R \setminus (U\cup V).
\end{equation*}
By compactness of $R \setminus (U\cup V)$ the sequence $\Phi(\tau, y)$ contains a sub-sequence whose limit is not in $U\cup V$.
\end{proof}

\begin{remark}
In the absence of compactness and positive invariance, the conclusions in Theorem~\ref{thm:first} fail to hold. For examples, in the ternary Cantor set, Smale horseshoe map, and the Poincare’ map of the isosceles three-body problem \citep{zare_chesley_1998},
R is not positively invariant and Theorem 1 does not apply.
\end{remark}

\section{The \texorpdfstring{$\omega$}{omega}-limit set in \texorpdfstring{$\mathbb{R}^n$}{R**n} for \texorpdfstring{$n\le3$}{n<=3}}\label{sec:ols_Rn}

\subsection{Possibilities with \texorpdfstring{$n=1$}{n=1} dimension}
This is a trivial case where the compact region $R$ is a closed interval and the \ols in $R$ is either a point (i.e., an equilibrium) or a closed sub-interval $R_1$ of $R$. By replacing $R$ by $R_1$, we have either an equilibrium in $R_1$ or a closed sub-interval $R_2$ of $R_1$. By repeating this argument we conclude that the \ols in $\real$ are {\em equilibrium points}.

\begin{example}\label{ex:first}
Let 
\begin{align*}
f(y) &= -y^5 + 15 y^4 -85 y^3 +225 y^2 - 274 y + 120\\
     &= (1-y)(2-y)(3-y)(4-y)(5-y)
\end{align*}
and $R=[0,6]$. As illustrated in Fig.~\ref{fig:n_one}, the \ols in $R$ are the three stable equilibria at $y=1$, $3$, and $5$.
\end{example}

\begin{figure}
\includegraphics[width=0.6\textwidth]{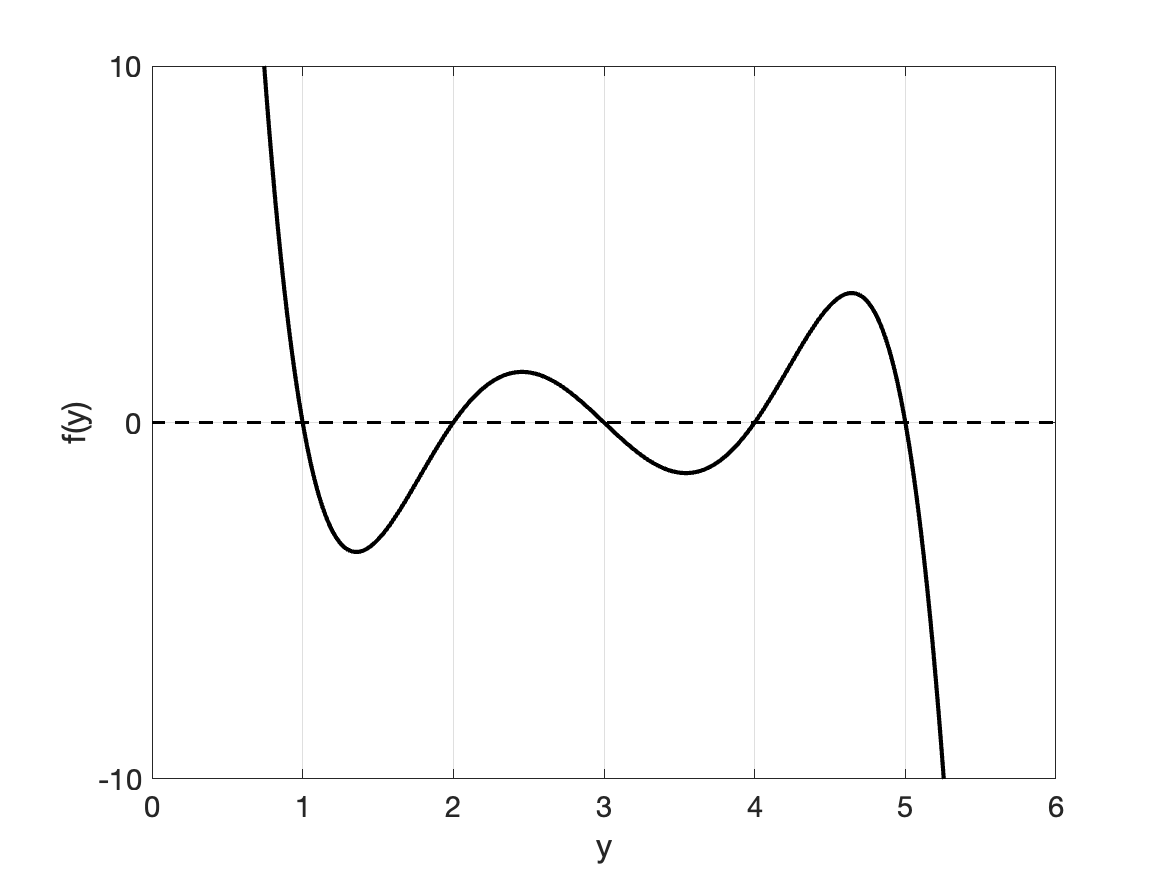}
\caption{\label{fig:n_one} The equilibria in the compact positively-invariant interval $R=[0,6]$ for Example~\ref{ex:first}.}
\end{figure}

\subsection{Possibilities with \texorpdfstring{$n=2$}{n=2} dimensions}

In this case the compact and positively-invariant region $R$ is a simply or multiply {\em connected area} in $\real^2$.
\begin{theorem}\label{thm:second}
The \ols in a compact and positively invariant region $R\subset\real^2$ can include only 
\begin{enumerate}
    \item equilibrium points,
    \item periodic orbits (i.e., limit cycles), and
    \item heteroclinic or homoclinic orbits
\end{enumerate}
\end{theorem}

\begin{proof}
The \ols of a trajectory in $R\subset\real^2$ is {\em nonempty}, {\em compact}, {\em invariant} and {\em connected} (see Theorem~\ref{thm:first}).The compact and connected sets in $R\subset\real^2$ are 
\begin{enumerate}[label=\alph*)]
    \item points
    \item closed cycles made of regular points, 
    \item closed cycles made of regular and irregular points, and
    \item a simply or multiply connected area $R_1\subset R$.
\end{enumerate}
However, in the case of d), by replacing $R$ with $R_1$, we have in $R_1$ the cases a), b) and c) or an area $R_2$. Repeating this argument leads to the conclusion that the compact and connected sets in $R$ are only a), b), and c). Since the \ols is also invariant, a), b), and c) are equivalent to 1, 2 and 3 in the statement of the theorem.
\end{proof}

\begin{corollary}\label{cor:second}
    If there is no equilibrium in $R\subset\real^2$ then there exists a limit cycle in $R$.
\end{corollary}

\begin{proof}
    If there is no equilibrium in $R$ then cases 1 and 3 in Theorem~\ref{thm:second} are not possible and the only possibility is case 2, indicating the existence of a limit cycle.
\end{proof}

\begin{remark}
    Note that Corollary~\ref{cor:second} is the well known Poincar\'e-Bendixon theorem, and Theorem~\ref{thm:second} leads to a different proof of this result.
\end{remark}

There are many physical examples in two dimensions. To mention a few, see the Van der Pol equation \citep{vdpol} model for an electrical network and the Brusselator model \citep{brusselator} for a chemical reaction.

\subsection{Possibilities with \texorpdfstring{$n=3$}{n=3} dimensions}

Here the compact and positively-invariant region $R$ is a simply or multiply {\em connected volume} in $\real^3$. 

In order to identify the \ols in $R$ we use the following {\em classification theorem} for orientable surfaces. 

\begin{theorem}[M\"obius theorem  \citep{mobius}]\label{thm:third}
     Every orientable surface is topologically equivalent to a sphere, a torus or a multiple torus (double, triple, etc.).
\end{theorem}

Although there are more advanced classification theorems that include nonorientable surfaces \citep{francis_weeks}, the above M\"obius theorem is sufficient for our purpose.

\begin{remark}
    The \ols on a sphere is the same as the \ols in $R\subset\real^2$ given in Theorem~\ref{thm:second}, i.e., the Euler numbers for a sphere and for $\real^2$ are the same, specifically, the Euler number is 2. However, the Euler number of a torus is zero, allowing orbits covering it densely with no equilibrium and no limit cycle.
\end{remark}

\begin{theorem}\label{thm:fourth}
The \ols in a compact and positively invariant region $R\subset\real^3$ can include only 
\begin{enumerate}
    \item equilibrium points,
    \item periodic orbits (i.e., limit cycles),
    \item heteroclinic or homoclinic orbits, and
    \item a surface topologically equivalent to a torus (single or multiple).
\end{enumerate}
\end{theorem}

\begin{proof}
The \ols of a trajectory in $R\subset\real^3$ is {\em nonempty}, {\em compact}, {\em invariant} and {\em connected} (see Theorem~\ref{thm:first}).The compact and connected sets in $R\subset\real^3$ are 
\begin{enumerate}[label=\alph*)]
    \item points 
    \item closed cycles made of regular points, 
    \item closed cycles made of regular and irregular points,
    \item surfaces topologically equivalent to a torus (single or multiple), and
    \item a simply or multiply connected region $R_1\subset R$.
\end{enumerate}
However, in the case of e), by replacing $R$ by $R_1$, in $R_1$ we have cases a), b), c), d) and a region $R_2\subset\real^3$. Repeating this argument leads to the conclusion that the compact and connected sets in $R$ are a), b), c) and d). Since the \ols is also invariant, a), b), c) and d) are equivalent to 1, 2, 3 and 4 in the statement of the theorem.
\end{proof}

To give an example in $\real^3$, we chose the well-known Lorenz equations discussed in the next section.

\section{The Lorenz attractor}\label{sec:lorenz}

A three-dimensional problem in meteorology that has been intensively studied is the Lorenz system of equations  \citep{lorenz}, which are given by
\begin{equation}
        \frac{dx}{dt} = 10(y-x);\qquad
    \frac{dy}{dt} = rx - y - xz;\qquad
    \frac{dz}{dt} = -\frac{8}{3}z + xy\ .
\end{equation}
The Lorenz equations model a layer of the atmosphere that is warmer at the bottom than at the top. The variables $x$, $y$, and $z$ are related, respectively, to the intensity of the air motion, the temperature variations in the horizontal and in the vertical directions of the atmospheric layer. The parameter $r>0$ is proportional to the temperature difference between the bottom and the top of the atmospheric layer. If $0<r<1$ then the origin is the only equilibrium and it is a {\em stable node}.

At $r=1$, the system exhibits a {\em pitchfork bifurcation}, i.e., the origin $\mathcal O$ becomes unstable and two new stable nodes are formed for $r>1$ at 
\begin{align*}
    {\mathcal A} &= ( -\sqrt{\frac{8}{3}(r-1)},\  -\sqrt{\frac{8}{3}(r-1)},\  r-1 ) \\
    {\mathcal B} &= ( +\sqrt{\frac{8}{3}(r-1)},\  +\sqrt{\frac{8}{3}(r-1)},\  r-1 ) .
\end{align*}

For $1 < r < 24.737$, $\mathcal A$ and $\mathcal B$ are stable equilibrium points, and they are the \ols of the Lorenz flow. For $r > 24.737$, $\mathcal A$ and $\mathcal B$ become unstable  and a new \ols is formed, which for $r=28$ has been obtained using numerical integration by  \citet{lorenz}. For a detailed discussion see  \citet{sparrow2012lorenz}.

To obtain a compact and positively invariant region $R$ for this problem, let $R$ be the interior of the sphere
$$
c^2 = x^2 + y^2 + (z-38)^2\ .
$$
It follows that 
\begin{align*}
    \frac{d}{dt}c^2 &= 2x\frac{dx}{dt} + 2y\frac{dy}{dt} + 2(z-38)\frac{dz}{dt} \\
                    &= -\frac{16}{3}\left[z^2 - 38z + \frac{3}{8}(10x^2 + y^2)\right] \ .
\end{align*}
The ellipsoidal surface for which $\frac{dc^2}{dt}=0$ can be written
\begin{align*}
    1 &= \frac{x^2}{\left(\frac{8(19)^2}{30}\right)} + \frac{y^2}{\left(\frac{8(19)^2}{3}\right)} + \frac{(z-19)^2}{19^2}
\end{align*}
 Within the ellipsoid (Fig.~\ref{fig:lorenz_sign}), $\frac{dc^2}{dt}$ is positive, while it is negative outside. But for $c \gtrsim 40$ the surface for $\frac{dc^2}{dt}=0$ is fully enclosed within the sphere defining $R$, and so we conclude that $\frac{dc^2}{dt}<0$ for any point on the sphere. Hence $R$ is a {\em positively invariant region}. The \ols obtained numerically is shown in Fig.~\ref{fig:ols}.

\begin{figure}
\includegraphics[width=\columnwidth]{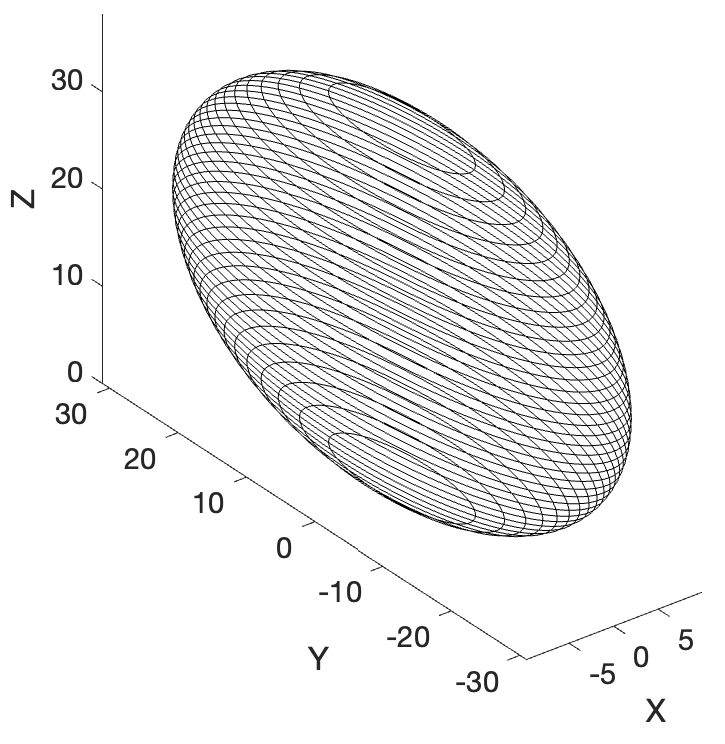}
\caption{\label{fig:lorenz_sign}  $\frac{dc^2}{dt}$ is negative outside (and positive inside) the depicted ellipsoid centered at $(0,0,19)$ and having semimajor axes $(\sqrt{\frac{2888}{30}}, \sqrt{\frac{2888}{3}}, 19) \simeq (9.8, 31.0, 19)$. }
\end{figure}

\begin{figure*}
\includegraphics[width=0.48\textwidth]{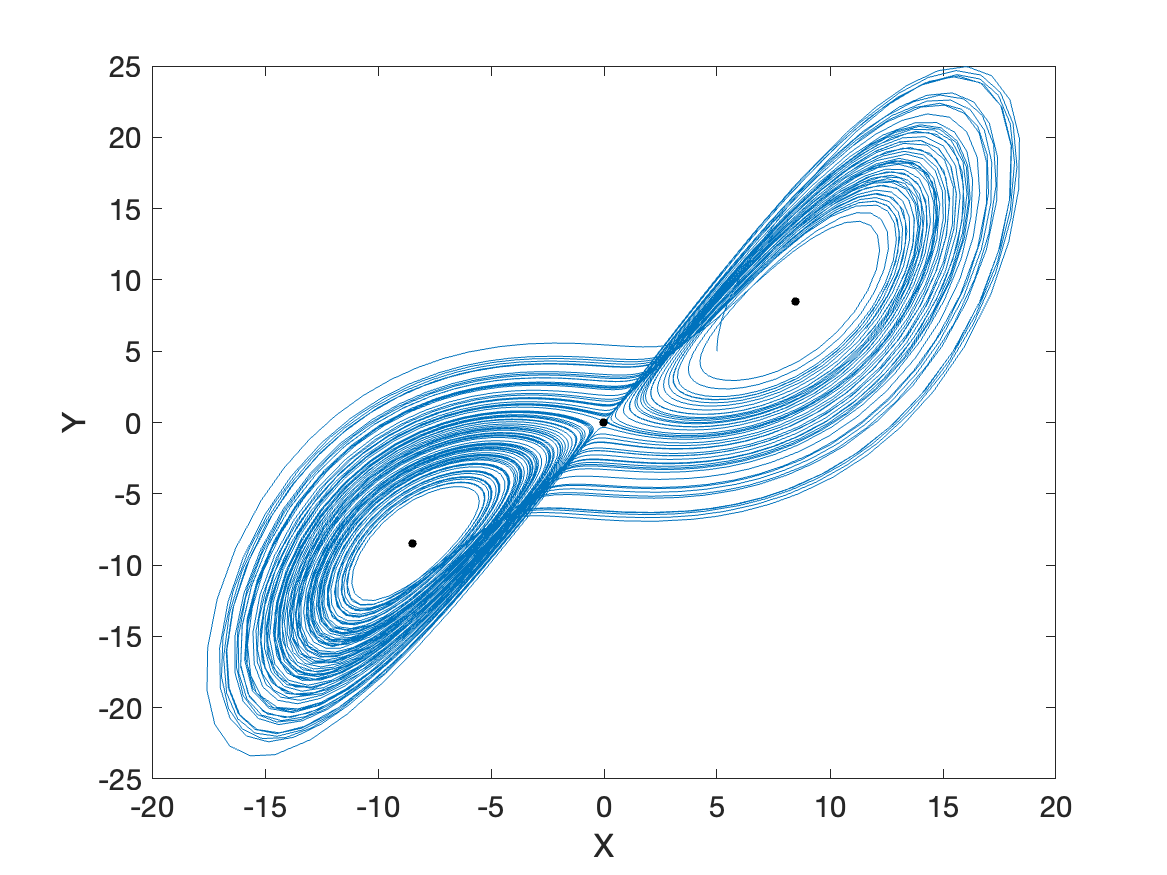}
\includegraphics[width=0.48\textwidth]{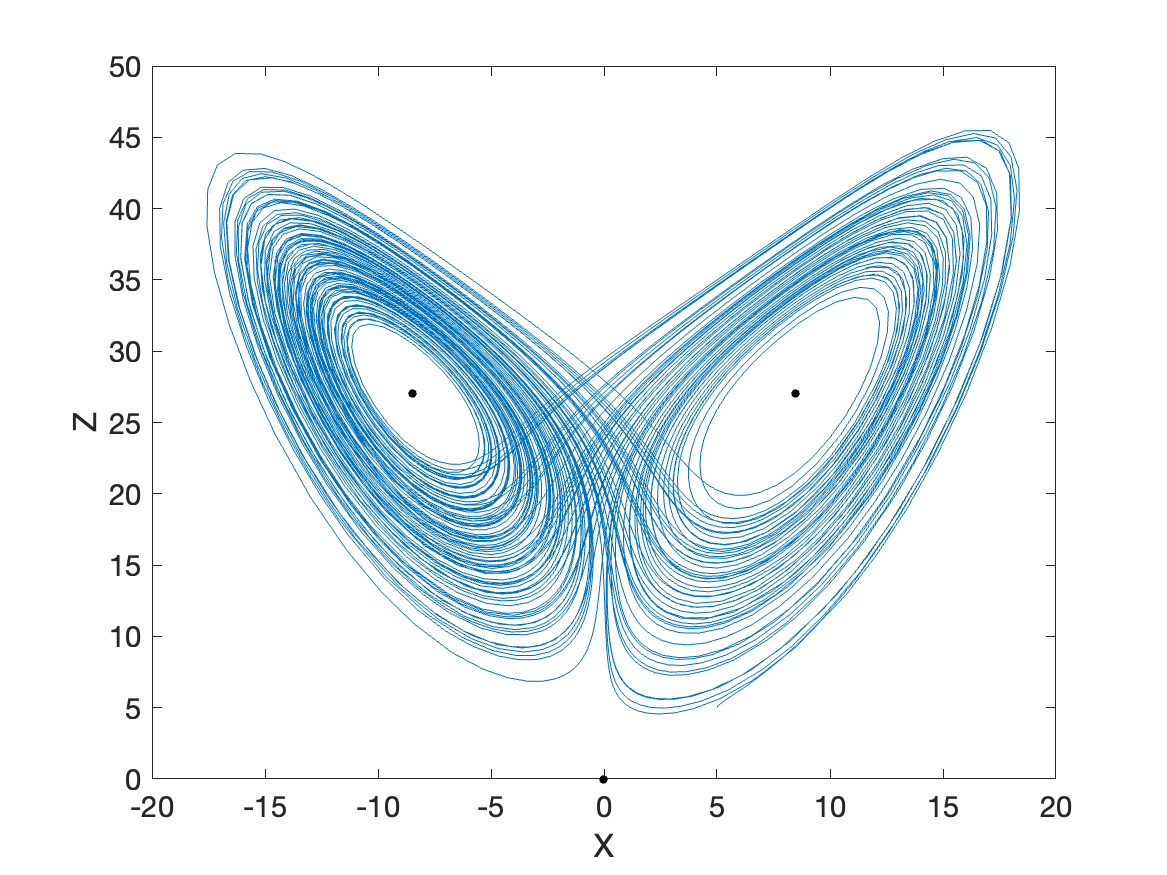}
\includegraphics[width=0.48\textwidth]{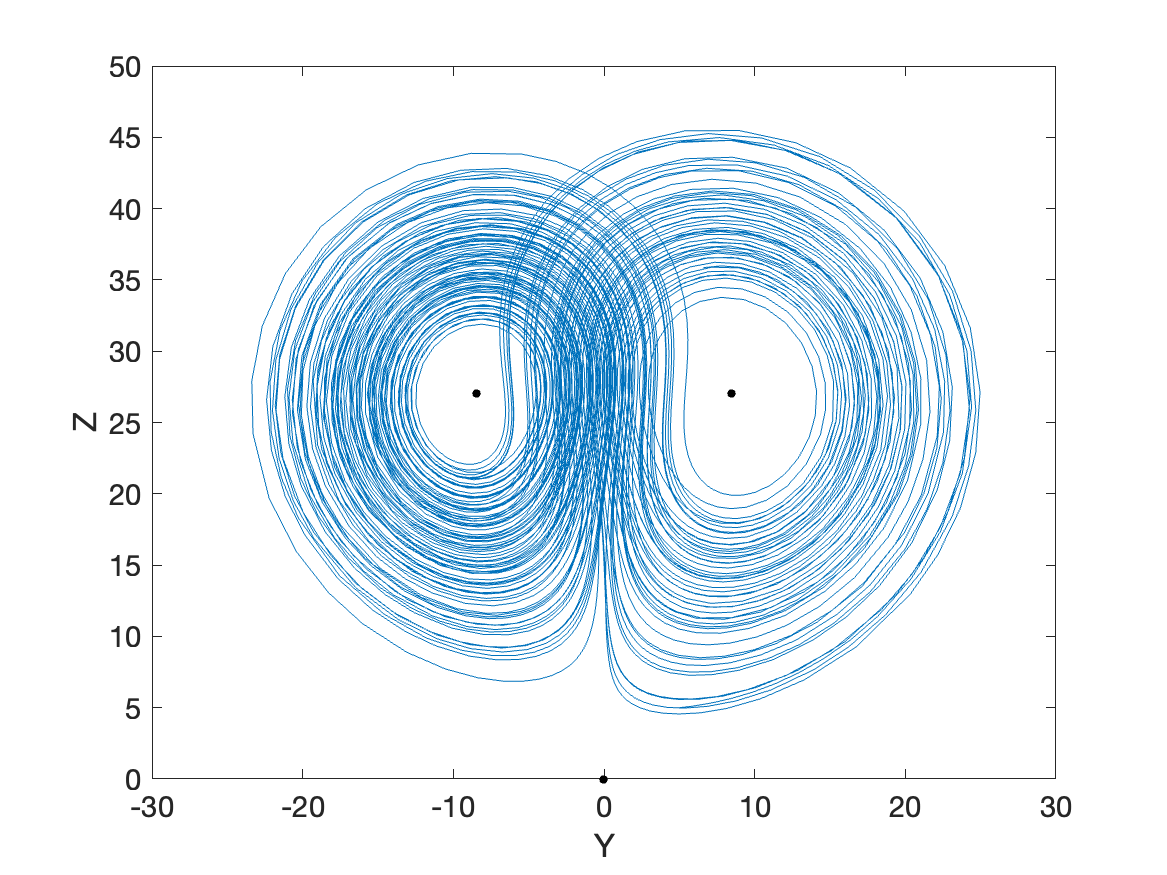}
\includegraphics[width=0.48\textwidth]{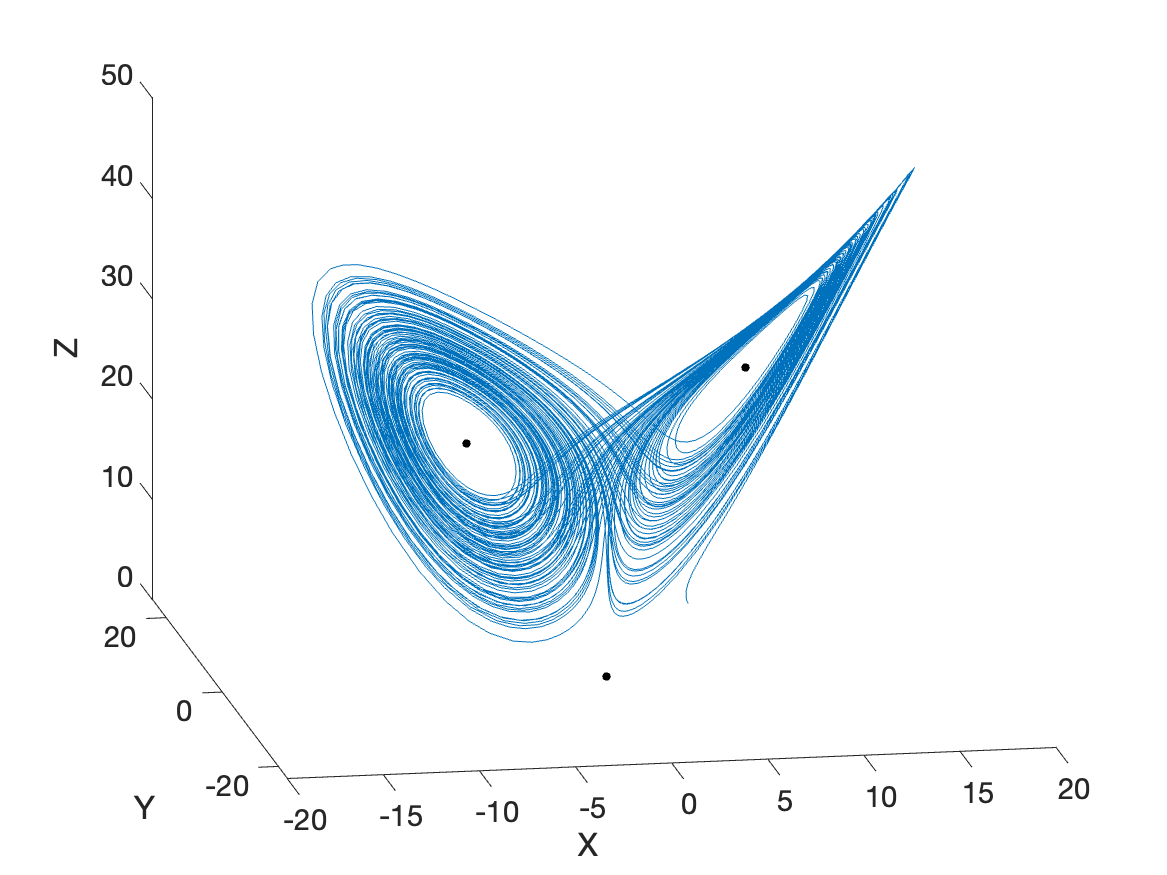}
\caption{\label{fig:ols} The \ols of the Lorenz flow using the initial condition $(x,y,z) = (5,5,5)$. Three two-dimensional projections are shown and an oblique three-dimensional representation. Equilibria $\mathcal A$, $\mathcal B$ and $\mathcal O$ are marked by black dots. }
\end{figure*}

We note that the equilibria $\mathcal A$, $\mathcal B$ and $\mathcal O$ are in $R$, but not on the \ols in Fig.~\ref{fig:ols}. We also observe that the numerically-obtained \ols shown in Fig.~\ref{fig:ols} resembles a twisted torus or a double torus. However, a double torus has an Euler number of -2, which implies it must have equilibrium points on it, but such equilibria do not exist. On the other hand, the Euler number of a twisted torus is zero, which allows no equilibria on it. This argument based on the Euler numbers leads to the conclusion that the \ols for the Lorenz attractor is a {\em twisted torus} 


With the exception of the eight stable manifolds of the equilibria---two each for $\mathcal A$ and $\mathcal B$  and four for $\mathcal O$---every other orbit in $R$ with arbitrary initial conditions approaches the \ols shown in Fig.~\ref{fig:ols}.

\begin{remark}
    The crossings of the $x=0$ plane have the property that the positions of consecutive crossing of the Lorenz orbit change regularly, while the times of consecutive crossing are chaotic.
\end{remark}



\section*{Acknowledgements}
The work of SRC was conducted at the Jet Propulsion Laboratory, California Institute of Technology, under a contract with the National Aeronautics and Space Administration. 

\nocite{*}
\bibliographystyle{aasjournal}
\bibliography{lorenz}

\end{document}